\documentclass[leqno]{article}
\usepackage{amsmath}
\usepackage{amsthm}
\usepackage{amsfonts}
\usepackage{bez123}

\newtheorem{lemma}{Lemma}
\newtheorem{theorem}{Theorem}

\theoremstyle{definition}
\newtheorem*{con}{Conjecture}
\newtheorem*{remark}{Remark}

\newcommand{\R}{\, \text{Re}\, }
\newcommand{\I}{\, \text{Im}\, }

\begin{document}
\title{Zeros of Dirichlet series with periodic coefficients}
\author{By Eric Saias and Andreas Weingartner}
\maketitle

\begin{abstract}
Let $a=(a_n)_{n\ge 1}$ be a periodic sequence, $F_a(s)$ the meromorphic continuation of 
$\sum_{n\ge 1} a_n/n^s$, and $N_a(\sigma_1, \sigma_2, T)$ the number of zeros of $F_a(s)$, 
counted with their multiplicities, in the rectangle $\sigma_1 < \R s < \sigma_2$, $|\I s | \le T$.
We extend previous results of Laurin\v{c}ikas, Kaczorowski, Kulas, and Steuding, by showing that if $F_a(s)$
is not of the form $P(s) L_{\chi} (s)$, where $P(s)$ is a Dirichlet polynomial and $L_{\chi}(s)$
a Dirichlet L-function, then there exists an $\eta=\eta(a)>0$ such that for all
$1/2 < \sigma_1 < \sigma_2 < 1+\eta$, we have
$c_1 T \le N_a(\sigma_1, \sigma_2, T) \le c_2 T$ for sufficiently large $T$, and suitable
positive constants $c_1$ and $c_2$ depending on $a$, $\sigma_1$, and $\sigma_2$.
\end{abstract}

\section{Introduction}

One of the most important open problems in mathematics is the

\begin{con} (Generalized Riemann Hypothesis)
Every $L_\chi(s)$ function, associated with a Dirichlet character $\chi$,
is zero-free in the open half-plane $\R (s) > 1/2$.
\end{con}

In this paper, we enlarge the set of $L_\chi (s)$ to Dirichlet series with
periodic coefficients, and investigate what can be shown in the
opposite direction about zeros in $\R (s)>1/2$. For the
distribution of zeros of these meromorphic functions in the
whole complex plane, we refer to Chapter 11 of the book by Steuding \cite{Steuding07}.

Let $q$ be a positive integer. Let $H_q$ be the $q$-dimensional Hilbert
space of Dirichlet series $\sum_{n\ge 1} a_n/n^s$, where $a=(a_n)_{n\ge 1}$ is a
$q$-periodic sequence of complex numbers, and with the scalar product
given by
\begin{equation}\label{one}
\langle  \sum_{n\ge 1} a_n/n^s , \sum_{n\ge 1} b_n/n^s \rangle
= \sum_{n=1}^q a_n \overline{b_n} .
\end{equation}
It is well known 
%(see for instance \cite{Steuding07}, Chapter 11) 
that these Dirichlet series $\sum_{n\ge 1} a_n/n^s$
have a meromorphic continuation to the entire complex plane with at most
one simple pole at $s=1$. We shall denote this meromorphic continuation by $F_a(s)$.

Let $\mathcal{D}_q^{\text{pr}}$ be the set of primitive characters that induce
the Dirichlet characters modulo $q$. For $\psi$ in $\mathcal{D}_q^{\text{pr}}$, we
denote by $E_{q, \psi}$ the subspace of $H_q$ generated by the functions
$L_\psi(s)/d^s$ where $d$ divides $q/\text{conductor}(\psi)$. 

We denote by $N_F (\sigma_1, \sigma_2, T)$ (respectively $N'_F (\sigma_1,
\sigma_2, T)$ ) the number of zeros of 
the function $F(s)$ in the rectangle 
$\sigma_1 < \R s < \sigma_2$, $|\I s | \le T$,
counted with their multiplicities (resp. without their multiplicities).

We begin with a structural theorem for $H_q$.

\begin{theorem}\label{thm1}
Let $q$ be a positive integer.
\begin{enumerate}
\item[(i)] 
The functions $L_\chi(s)/d^s$, where $d$ runs through the divisors of $q$,
and $\chi$ is a Dirichlet character modulo $q/d$, 
form an orthogonal basis of $H_q$.
\item[(ii)] We have the orthogonal decomposition
\begin{equation*}
H_q = \bigoplus_{\psi \in \mathcal{D}_q^{\text{pr}}} E_{q,\psi}
\end{equation*}
\end{enumerate}
\end{theorem}

Thus every function $F_a(s)$ can be written in a unique way as 
\begin{equation}\label{two}
    F_a(s) = \sum_{\psi \in \mathcal{D}_q^{\text{pr}}} P_\psi(s) L_{\psi} (s) 
\end{equation}
where the $P_\psi(s)$ are Dirichlet polynomials that 
satisfy certain specific conditions. Ignoring these conditions for the moment,
we get a much larger set of functions, for which we have the following result.

\begin{theorem}\label{thm2}
Let $ \mathcal{C}$ be a finite set of at least two primitive Dirichlet characters,
and let $(P_\psi)_{\psi \in \mathcal{C}}$ be a family of non-zero Dirichlet polynomials.
Define
\begin{equation*}
F(s) := \sum_{\psi \in \mathcal{C}} P_\psi(s) L_\psi(s) .
\end{equation*}
Then there exists a number $\eta=\eta(F)>0$ such that,
for all real numbers
$\sigma_1$ and $\sigma_2$ with $1/2 \le \sigma_1 < \sigma_2 \le 1+\eta$
and all sufficiently large $T$, we have
\begin{equation*}
    N'_F(\sigma_1,\sigma_2, T) \gg_{F, \sigma_1, \sigma_2} T .
\end{equation*}
\end{theorem}

For the upper bound for the number of zeros, we come back to the
smaller set of Dirichlet series with periodic coefficients.

\begin{theorem}\label{thm3}
Let $a=(a_n)_{n\ge 1}$ be a non-zero periodic sequence. Then
\begin{equation*}
N_{F_a}\left( \frac{1}{2}+u, +\infty, T \right) \ll_a T \frac{\log(1/u)}{u}
\end{equation*}
for $0 < u \le 1/2$ and $T\ge 1$.
\end{theorem}

Combining these three results, we finally get the result that motivated
this paper.

\begin{theorem}\label{thm4}
Let $q\ge 1$. Let $a=(a_n)_{n\ge 1}$ be a $q$-periodic sequence such that
$\sum_{n\ge 1} a_n/n^s$ does not belong to one of the subspaces $E_{q,\psi}$,
$\psi \in \mathcal{D}_q^{\text{pr}}$. 
Then there exists a number $\eta = \eta(a) >0$ such that, for all real numbers
$\sigma_1$ and $\sigma_2$ with $1/2 < \sigma_1 < \sigma_2 \le 1+\eta$
and all sufficiently large $T$, we have
\begin{equation*}
    N_{F_a}(\sigma_1,\sigma_2, T) \asymp_{a, \sigma_1, \sigma_2} 
    N'_{F_a}(\sigma_1,\sigma_2, T) \asymp_{a, \sigma_1, \sigma_2}
    T .
\end{equation*}
\end{theorem}

\begin{remark}
It follows that, if $F_a(s)$ does not vanish in $\R s > 1/2$, then
\begin{equation*}
    F_a(s) = P(s) L_\psi (s)
\end{equation*}
for some primitive Dirichlet character $\psi$ and some Dirichlet polynomial $P(s)$.
Thus the functions $L_\psi (s)$ with $\psi \in \mathcal{D}_q^{\text{pr}}$ 
turn out to be a kind of ``primitive function'' for all those $F_a(s)$ with 
$q$-periodic $a=(a_n)_{n\ge 1}$, which conjecturally do not vanish in $\R s > 1/2$.
\end{remark}

We recall that the Dirichlet characters are exactly the arithmetic
functions which are both periodic and completely multiplicative. What
are the roles of these two properties for the Generalized Riemann Hypothesis (GRH)? 
What we find here about the zeros of Dirichlet series with periodic coefficients, 
confirms the commonly held idea that in any proof of GRH, the Euler Product, which 
comes from complete multiplicativity, must play a significant role.

Theorem \ref{thm1} follows easily from the orthogonal basis 
of $q$-periodic sequences canonically associated with Dirichlet
characters modulo $q$, which has been used, and perhaps discovered, 
by Codec\`{a}, Dvornicich, and Zannier \cite{CDZ98}, Lemma 1.

The case $\eta=0$ in Theorem \ref{thm2} follows from \cite{KK07}, Theorem 2,
of Kaczorowski and Kulas.
They use the classical way to get zeros off the critical line, which is
to apply a strong joint universal property for the Dirichlet L-functions.
But, as far as we know, that method requires that one works within the strip 
$1/2 < \R s < 1$.
To get zeros in the half-plane $\R s > 1$, we use here a kind of weak joint universal 
property for the Dirichlet L-functions. This leads us to add a new tool into the 
picture: the Brouwer fixed point theorem (see Lemma \ref{lem2}).

Theorem \ref{thm3} is an explicit form of the upper bound of Steuding
\begin{equation}\label{SteuInt}
N_{F_a}(1/2+u, +\infty, T) \ll_{a,u} T.
\end{equation}
More precisely, the proof of the slightly weaker 
$N_{F_a}(1/2+u, +\infty, T) \ll_{a,u} T\log T$ 
appears in \cite{Steu}. 
In \cite{Steuding07}, the upper bound \eqref{SteuInt} is stated in Theorem 11.3, 
but the proof
is given only in the analog situation of the extended Selberg class.
For the sake of completeness, we give here the details of the proof in our
situation of Dirichlet series with periodic coefficients, and take the
opportunity to make the dependence on $u$ explicit.

For Theorem \ref{thm4}, the lower bound 
$N_{F_a} (\sigma_1, \sigma_2,T) \gg T$
appears in the paper of Laurin\v{c}ikas \cite{Lauri} with the condition $1/2
< \sigma_1 < \sigma_2< 1$, and the restriction that
the sequence $a$ be a linear combination of at least 
two Dirichlet characters modulo $q$. 

\newpage

\section{Proof of Theorem 1}

(i) For a Dirichlet character $\chi$ modulo $q/d$, we denote by $\widetilde{\chi}$
the arithmetic function defined by
\begin{equation*}
    \widetilde{\chi} = \chi \left(\frac{n}{d}\right)
\end{equation*}
with the usual convention that $\chi(t)=0$ if $t$ is not a positive integer.
By Lemma 1 of \cite{CDZ98}, the functions $\widetilde{\chi}$ form an orthogonal
basis for the $q$-periodic sequences $(a_n)_{n\ge 1}$ with scalar product 
$\langle a,b \rangle = \sum_{n=1}^q a_n \overline{b_n}$. The result now follows from 
\begin{equation*}
  \sum_{n=1}^{+\infty} \frac{\widetilde{\chi}(n)}{n^s} = \frac{L_\chi(s)}{d^s},
  \qquad   \R s > 1 .
\end{equation*}

(ii) We are going to apply part (i) and a change of basis in each $E_{q,\psi}$. 
Let $\psi$ be a primitive Dirichlet character whose conductor $m$ is a divisor of $q$.
For a Dirichlet character $\chi$ modulo $q/d$ induced by $\psi$ we have
\begin{equation*}
\begin{split}
\frac{L_\chi(s)}{d^s L_\psi(s)} &= \frac{1}{d^s}\prod_{p | \frac{q}{d}}
\left(1-\frac{\psi(p)}{p^s}\right) \\
 &= \frac{1}{d^s}\prod_{p | \frac{q}{m d}}
\left(1-\frac{\psi(p)}{p^s}\right) \\
&= \frac{1}{q'^s} b^s \prod_{p | b}
\left(1-\frac{\psi(p)}{p^s}\right) 
\end{split}
\end{equation*}
where $q' := \frac{q}{m}$ and $b := \frac{q'}{d}$.

By part (i) of the theorem, we thus have the orthogonal sum
\begin{equation}\label{Hq}
    H_q = \bigoplus_{\psi \in \mathcal{D}_q^{\text{pr}}} 
    \frac{L_\psi(s)}{q'^s} \cdot V_{q,\psi} ,
\end{equation}
where
\begin{equation*}
\begin{split}
V_{q,\psi} &= \text{Vect} \left\{b^s \prod_{p | b}
\left(1-\frac{\psi(p)}{p^s}\right) : b | q' \right\} \\
 &= \text{Vect} \left\{\left( \prod_{i=1}^r x_i^{\beta_i}\right)
 \left(\prod_{\beta_i \ge 1} \left(1-\frac{\psi(p_i)}{x_i}\right)\right) : 
 0 \le \beta_i \le \alpha_i \right\} 
\end{split}
\end{equation*}
with
\begin{equation*}
    q' = p_1^{\alpha_1} \cdots p_r^{\alpha_r}
    \quad \text{and} \quad
    x_i := p_i^s .
\end{equation*}
We order the two families
\begin{enumerate}
\item the free family of 
$\displaystyle \mathbb{C} [x_1,\ldots, x_r] : 
\left(\prod_{i=1}^r x_i^{\beta_i}\right)_{0 \le \beta_i \le \alpha_i}$
\item the family 
$\displaystyle \left( \prod_{i=1}^r x_i^{\beta_i}\right)
 \left(\prod_{\beta_i \ge 1} \left(1-\frac{\psi(p_i)}{x_i}\right)
 \right)_{0 \le \beta_i \le \alpha_i }$
\end{enumerate}
according to the order on the $\beta=(\beta_i)$ given by
\begin{equation*}
    \beta < \beta' \ \ \text{iff} \ \ \left|
    \begin{array}{l}
    \displaystyle \sum_{i=1}^r \beta_i < \sum_{i=1}^r \beta'_i \\
    \qquad \text{or} \\
    \displaystyle  \sum_{i=1}^r \beta_i = \sum_{i=1}^r \beta'_i
    \ \ \text{and there is a $j$ with} \ \ \left|
    \begin{array}{l}
    i<j \Rightarrow \beta_i=\beta'_i \\
    \beta_j < \beta'_j
    \end{array} \right.
    \end{array} \right.
\end{equation*}
We observe that the second family is then the image of the first family under
an upper triangular matrix with ones on the diagonal. Thus
\begin{equation*}
    \frac{V_{q,\psi}}{q'^s} 
    = \frac{1}{q'^s} \text{Vect}\{d^s : d|q' \} 
    = \text{Vect}\left\{\frac{1}{d^s} : d|q' \right\},
\end{equation*}
and the result follows from \eqref{Hq}.

\section{Preparation for Theorem 2}

In the following two lemmas, we use the notation
\begin{equation*}
    D_n(R):= \left\{ z=(z_j)_{1\le j \le n} \in \mathbb{C}^n :
    |z_j|\le R \ \text{ for all } \ 1\le j \le n \right\}.
\end{equation*}

\begin{lemma}\label{lem1}
Let $q$ be a positive integer, and $y$ and $R$ be positive real numbers.
Let $\chi_1,\ldots,\chi_n$ be pairwise distinct Dirichlet characters modulo $q$.
Then there exists a real $\eta > 0$ such that for all fixed $\sigma$ with 
$1<\sigma \le 1+\eta$, and for all prime numbers $p>y$, there exists a continuous
function $t_p : D_n(R)\longrightarrow \mathbb{R}$, such that for all $z$ in $D_n(R)$
\begin{equation*}
    z=\left( \sum_{p>y} \frac{\chi_j(p)}{p^{ \, \sigma+i t_p(z)}}\right)_{1\le j \le n}
\end{equation*}
\end{lemma}

\begin{remark}
We can interpret this lemma as a linear system to be solved. There are $n$
equations. The unknowns are the infinite family of $\left(p^{-it_p}\right)_{p>y}$
that must be chosen in the unit circle. The $z \in \mathbb{C}^n$ is a parameter.
Moreover, the solution must be chosen continuously in the parameter $z$.
\end{remark}

\begin{proof}
If $n<\varphi(q)$, we extend $(\chi_j)_{1\le j \le n}$ to $(\chi_j)_{1\le j \le \varphi(q)}$,
using all the Dirichlet characters modulo $q$. This allows us to restrict the proof
to the case $n=\varphi(q)$.

We denote by $C$ the unitary matrix of the characters modulo $q$. That is, 
\begin{equation*}
C:= \left( \chi_j(a)\right)_{\substack{1\le a \le q, \ (a,q)=1 \\ 1\le j \le \varphi(q)
  \qquad}}
\end{equation*}
We have 
\begin{equation*}
   \sum_{p>y} \frac{\chi_j(p)}{p^{ \, \sigma+i t_p}} 
   =\sum_{\substack{1\le a \le q \\ (a,q)=1}} \chi_j(a) 
     \sum_{\substack{p>y \\ p\equiv a \, (q)}} \frac{1}{p^{ \, \sigma+i t_p}} 
\end{equation*}
To change variables we write
\begin{equation*}
 z=C w ,
\end{equation*}
where 
\begin{equation*}
z=(z_j)_{1\le j \le \varphi(q)} \quad \text{and} \quad 
w=(w_a)_{\substack{1\le a \le q \\ (a,q)=1}} ,
\end{equation*}
and
\begin{equation*}
 \theta_p = -(\log p) ( t_p \circ C) .
\end{equation*}
To prove the lemma, it is sufficient to solve the system
\begin{equation}\label{weq}
    \sum_{\substack{p>y \\ p\equiv a \, (q)}} \frac{e^{i \theta_p}}{p^\sigma}= w_a,
     \qquad
    1\le a \le q, \ (a,q)=1,
\end{equation}
in the real unknowns $(\theta_p)_{p>y}$, 
continuously in $w \in D_{\varphi(q)}(\|C^{-1}\|_\infty R)$.
We put
\begin{equation*}
 S_a = S_a(q,y,\sigma) := \sum_{\substack{p>y \\ p\equiv a \, (q)}} \frac{1}{p^\sigma} .
\end{equation*}
Using the prime number theorem for arithmetic progressions, we readily find that there exists an 
$\eta > 0$, such that for each $1<\sigma \le 1+\eta$ and $1\le a \le q$, $(a,q)=1$, we have
\begin{equation}\label{Sa}
    S_a \ge 10 \|C^{-1}\|_\infty R ,
\end{equation}
and there exist prime numbers $p_{1,a}$ and $p_{2,a}$, such that
\begin{equation*}
\frac{1}{3} \le \lambda_0 
:= \frac{1}{S_a} \sum_{\substack{y< p \le p_{1,a} \\ p\equiv a \, (q)}} \frac{1}{p^\sigma}
\le \frac{1}{3} +  \frac{1}{100} 
\end{equation*}
and
\begin{equation*}
\frac{1}{3} \le \lambda_1
:= \frac{1}{S_a} \sum_{\substack{p_{1,a}< p \le p_{2,a} \\ p\equiv a \, (q)}} \frac{1}{p^\sigma}
\le \frac{1}{3} +  \frac{1}{100} .
\end{equation*}
We also write
\begin{equation*}
\lambda_2
:= \frac{1}{S_a} \sum_{\substack{p > p_{2,a} \\ p\equiv a \, (q)}} \frac{1}{p^\sigma},
\end{equation*}
such that 
\begin{equation*}
\lambda_0 + \lambda_1 + \lambda_2 =1 .
\end{equation*}
We choose
\begin{equation*}
\theta_p = \left|
\begin{array}{ll}
0           & \ \ \text{if} \ \ y<p\le p_{1,a} \\
\pi + u_1   & \ \ \text{if} \ \ p_{1,a}<p \le p_{2,a} \\
\pi - u_2   & \ \ \text{if} \ \ p_{2,a}<p \\
\end{array} \right.
\end{equation*}
with $u_1$ and $u_2$ to be fixed later. In view of \eqref{weq}
it is sufficient to solve, for each $a$, the equation
\begin{equation}\label{leq}
    \lambda_1 e^{i u_1} + \lambda_2 e^{-i u_2} = \lambda_0 - \frac{w_a}{S_a}
\end{equation}
in the real unknowns $u_1$ and $u_2$, continuously in $w_a$ for $|w_a|\le \|C^{-1}\|_\infty R$.
We define the function $F$ by
\begin{equation*}
\begin{split}
F : \ \left]0, \frac{\pi}{2} \right[^2  & \longrightarrow \mathbb{C} \\
(u_1, u_2) & \longmapsto \lambda_1 e^{i u_1} + \lambda_2 e^{-i u_2} .
\end{split}
\end{equation*}
$F$ is a diffeomorphism onto its image. Moreover, since
$\frac{1}{3} \le \lambda_0, \lambda_1 \le \frac{1}{3}+\frac{1}{100}$, 
and
$\frac{1}{3} - \frac{1}{50}\le \lambda_2 \le \frac{1}{3}$,
we have
\begin{equation*}
\left\{ s \in \mathbb{C} : |s-\lambda_0| \le \frac{1}{10} \right\} \subset \text{ Im } F,
\end{equation*}
as illustrated in the following figure.

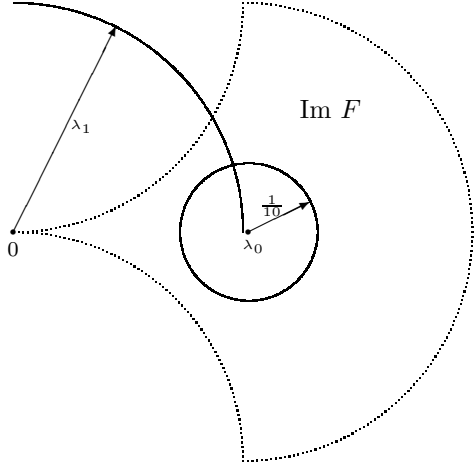
\begin{figure}[h]
\setlength{\unitlength}{.3in}
\noindent
\begin{picture}(14,8)
% 4 units represent a length of 1/3 = lambda_0
  \put(4,4){\circle*{.1}}
  \put(3,3.2){\makebox(2,1){{\footnotesize 0}}}
  \put(4,4){\vector(1,2){1.8}}
  \put(5,5.8){\tiny $\lambda_1$}
  \cbezier(4,8)(6.2091389992,8)(8,6.2091389992)(8,4)
  \cbezier[80](12,4)(12,1.7908610008)(10.2091389992,0)(8,0)
  \cbezier[80](8,8)(8,5.7908610008)(6.2091389992,4)(4,4)
  \cbezier[80](8,8)(10.2091389992,8)(12,6.2091389992)(12,4)
  \cbezier[80](4,4)(6.2091389992,4)(8,2.2091389992)(8,0)
  
  \cbezier(6.9,4)(6.9,4.662742)(7.437258,5.2)(8.1,5.2)
  \cbezier(8.1,5.2)(8.762742,5.2)(9.3,4.662742)(9.3,4)
  \cbezier(9.3,4)(9.3,3.337258)(8.762742,2.8)(8.1,2.8)
  \cbezier(8.1,2.8)(7.437258,2.8)(6.9,3.337258)(6.9,4)
  \put(8.1,4){\vector(2,1){1.08}}
  \put(8.3,4.4){\tiny $\frac{1}{10}$}
  \put(8.1,4){\circle*{.1}}
  \put(8,3.7){\tiny $\lambda_0 $}
  
  \put(9,6){Im $F$}
\end{picture}
\caption{The image of $F$, depicted by the region with the dotted boundary, 
contains the disk with center $\lambda_0$ and radius $\frac{1}{10}$.}
\end{figure}

Thus by \eqref{Sa} we can solve \eqref{leq} continuously in $w_a$. This
concludes the proof of Lemma \ref{lem1}.
\end{proof}

\newpage

\begin{lemma}\label{lem2}
Let $q$ and $L$ be positive integers, and $R\ge 1$ be real.
Let $\chi_1, \ldots, \chi_n$ be pairwise distinct Dirichlet characters modulo $q$.
For all $1\le j \le n$, let $h_j$ be a non-zero rational function in $L$ complex variables.
Then there exists a real $\eta>0$ such that, for all $\sigma$ with $1<\sigma \le 1+\eta$, we have
\begin{multline*}
\left\{ z \in \mathbb{C}^n : \frac{1}{R} \le |z_j| \le R \right\} \\
\subset
\left\{ \left( h_j\left( \frac{1}{p_1^{\sigma+it_{p_{1}}}}, \ldots,\frac{1}{p_{L}^{\sigma+it_{p_{L}}}}\right)
\prod_{p>p_L} \left( 1-\frac{\chi_j(p)}{p^{\, \sigma+i t_p}} \right)^{-1}
\right)_{1\le j \le n} : t_p \in \mathbb{R} \right\}
\end{multline*}
\end{lemma}

\begin{proof}
We first consider the particular case where all the $h_j$ are $1$. We put $y=p_L$
and $R'=\pi+\log R$.
Applying Lemma \ref{lem1} (and changing the letter $z$ to $w$) we have continuous 
functions $t_p$ such that
\begin{equation}\label{lem2eq1}
    w_j = \sum_{p>y} \frac{\chi_j(p)}{p^{\, \sigma + it_p(w)}} ,
    \qquad 
    w \in D_n(1+R'), \ 1\le j \le n .
\end{equation}
We define the error term $E$ by
\begin{equation}\label{lem2eq2}
    \left(\sum_{p>y} \log \left(1-\frac{\chi_j(p)}{p^{\, \sigma +it_p}}\right)\right)_{1\le j\le n}
    = \left( -\sum_{p>y} \frac{\chi_j(p)}{p^{\, \sigma +it_p}} \right)_{1\le j\le n}
    + E\left( (t_p)_{p>y} \right) .
\end{equation}
The real number $\sigma>1$ being fixed, the function $E$ is continuous for the weak-convergence
topology. Moreover, for all $j$ and all $(t_p)_{p>y}$, we have
\begin{equation}\label{lem2eq3}
    \left| E_j \left( (t_p)_{p>y} \right) \right| \le \sum_p \frac{1}{p^2} <1 .
\end{equation}
Let $z\in D_n(R')$ be fixed. 
From \eqref{lem2eq3} we see that, for all $j$ and all $(t_p)_{p>y}$,
\begin{equation*}
    \left| z_j + E_j \left( (t_p)_{p>y} \right) \right| \le 1+R' .
\end{equation*}
Thus we have the following continuous function.
\begin{equation*}
\begin{split}
F : \  D_n(1+R') & \longrightarrow  D_n(1+R') \\
w & \longmapsto z + E\left( (t_p(w))_{p>y} \right)
\end{split}
\end{equation*}
The Brouwer fixed point theorem shows that there exists a $w \in D_n(1+ R')$
such that $F(w)=w$. Together with \eqref{lem2eq1} and \eqref{lem2eq2} this yields
 \begin{equation*}
   \left(-\sum_{p>y} \log \left(1-\frac{\chi_j(p)}{p^{\, \sigma +it_p(w)}}\right)
   \right)_{1\le j\le n} = z .
\end{equation*}
Taking exponentials allows us to conclude the case when $h_j \equiv 1$.

We now consider the case with a general $h$. Let us choose 
$\left( t_{p_1}, \ldots, t_{p_L}\right)$ such that for all j, 
$h_j\left( \frac{1}{p_1^{\sigma+it_{p_{1}}}}, \ldots,\frac{1}{p_{L}^{\sigma+it_{p_{L}}}}\right)$
has neither zeros nor poles for $1\le \sigma \le 2$.
We put
\begin{equation*}
\begin{split}
c &:= \min_{1\le j \le n} \min_{1\le \sigma \le 2}\left|
h_j\left( \frac{1}{p_1^{\sigma+it_{p_{1}}}}, \ldots,\frac{1}{p_{L}^{\sigma+it_{p_{L}}}}\right)
\right| ,\\
C &:= \max_{1\le j \le n} \max_{1\le \sigma \le 2}\left|
h_j\left( \frac{1}{p_1^{\sigma+it_{p_{1}}}}, \ldots,\frac{1}{p_{L}^{\sigma+it_{p_{L}}}}\right)
\right| .
\end{split}
\end{equation*}
Applying the particular case where $h_j \equiv 1$ with 
$\widetilde{R}= \max \left(\frac{C}{R}, \frac{R}{c}\right)$ allows us to conclude the general case.
\end{proof}

\section{Proof of Theorem 2}

If $\sigma_1 <1$ then $N'_F(\sigma_1, \sigma_2, T) \gg_{F, \sigma_1, \sigma_2} T$ 
by Theorem 2 of \cite{KK07}. We may thus restrict our attention to the case $\sigma_1 \ge 1$.

We choose $q$ to be the least common multiple of the conductors of the $\psi$ in $\mathcal{C}$,
and we write $\mathcal{C} = \{\psi_1, \ldots, \psi_n\}$ with $2\le n \le \varphi(q)$.
We use the notation
\begin{equation*}
  F_j(s) = P_{\psi_j}(s) L_{\psi_j}(s)
\end{equation*}
and
\begin{equation*}
  P_{\psi_j}(s) = \sum_{k\ge 1} \frac{c_{j,k}}{k^s}
\end{equation*}
We choose $y=p_L$ such that if $p$ divides a $k$ for which there is a $j$ such
that $c_{j,k} \neq 0$, then $p\le y$. Denoting by $\chi_j$ the Dirichlet character
modulo $q$ that is induced by $\psi_j$ we can thus write
\begin{equation*}
F_j(s) = 
h_j\left( \frac{1}{p_1^{s}}, \ldots,\frac{1}{p_{L}^{s}}\right)
\prod_{p>p_L} \left( 1-\frac{\chi_j(p)}{p^{\, s}} \right)^{-1}
\end{equation*}
where $h_j$ is a nonzero rational function such that 
\begin{equation}\label{thm3eq1}
  h_j \ \text{ has no poles in } \
  \{ (z_1, \ldots, z_L) \in \mathbb{C}^L : |z_l|<1 \} .
\end{equation}
Choosing $R=1$ we get by Lemma \ref{lem2} a real $\eta>0$, 
which will be the one we use for Theorem \ref{thm2}.  
Let $\sigma_1$ and $\sigma_2$ be real numbers such that $1\le \sigma_1 < \sigma_2 \le 1+\eta$.
We choose
\begin{equation*}
  \sigma = \frac{\sigma_1 +\sigma_2}{2} .
\end{equation*}
By Lemma \ref{lem2}, there is a sequence $(t_p)_p$ of real numbers such that for all $j$,
$1\le j \le n$, 
\begin{equation*}
  h_j\left( \frac{1}{p_1^{\sigma+it_{p_{1}}}}, \ldots,\frac{1}{p_{L}^{\sigma+it_{p_{L}}}}\right)
\prod_{p>p_L} \left( 1-\frac{\chi_j(p)}{p^{\, \sigma+i t_p}} \right)^{-1}
=e^{2i\pi j/n}
\end{equation*}
We write
\begin{equation*}
  G_j(s) := h_j\left( \frac{1}{p_1^{s+it_{p_{1}}}}, \ldots,\frac{1}{p_{L}^{s+it_{p_{L}}}}\right)
\prod_{p>p_L} \left( 1-\frac{\chi_j(p)}{p^{\, s+i t_p}} \right)^{-1} .
\end{equation*}
As $n \ge 2$, we have
\begin{equation}\label{thm3eq2}
    \sum_{j=1}^n G_j(\sigma) = 0 .
\end{equation}
We now choose a circle $C=C(\sigma,r)$ centered at $ \sigma = \frac{\sigma_1 +\sigma_2}{2}$
and with a radius $r$ with $0 < r < \frac{\sigma_2-\sigma_1}{2}$, such that 
$ \sum_{j=1}^n G_j(s)$ does not vanish on $C$. We write 
\begin{equation*}
\gamma := \min_{s \in C} \left| \sum_{j=1}^n G_j(s) \right| > 0 .
\end{equation*}
Because of \eqref{thm3eq1} and the uniform convergence of the infinite products, we can
choose a prime number $p_M \ge p_L$ such that for all $j$, $1\le j \le n$,
\begin{equation*}
\left| F_j(z) - 
h_j\left( \frac{1}{p_1^{z}}, \ldots,\frac{1}{p_{L}^{z}}\right)
\prod_{p_L < p \le p_M} \left( 1-\frac{\chi_j(p)}{p^{\, z}} \right)^{-1} \right|
< \frac{\gamma}{3n} , \qquad \R z \ge \sigma-r,
\end{equation*}
and
\begin{equation*}
 \left| G_j(s) - h_j \! \! 
 \left( \frac{1}{p_1^{s+it_{p_{1}}}}, \ldots,\! \frac{1}{p_{L}^{s+it_{p_{L}}}} \! \! \right)
\! \! \! \prod_{p_L <p \le p_M} \! \! \! 
\left( 1-\frac{\chi_j(p)}{p^{\, s+i t_p}} \right)^{-1} \right|
< \frac{\gamma}{3n}, \quad \R s \ge \sigma-r .
\end{equation*}

By Weyl's criterion, we know that the set $\{p_1^{it},\ldots,p_M^{it}\}$ is uniformly
distributed in $\{z : |z|=1 \}^M$. 
Using \eqref{thm3eq1} once more it follows that the set of $t \in \mathbb{R}$, such
that for all $s$ with $|s-\sigma| \le r$ and all $j$, $1\le j \le n$, 
\begin{multline*}
 \left| 
 h_j\left( \frac{1}{p_1^{s+it}}, \ldots,\frac{1}{p_{L}^{s+it}}\right)
\prod_{p_L <p \le p_M} \left( 1-\frac{\chi_j(p)}{p^{\, s+i t}} \right)^{-1} \right. \\
\left. - h_j\left( \frac{1}{p_1^{s+it_{p_{1}}}}, \ldots,\frac{1}{p_{L}^{s+it_{p_{L}}}}\right)
\prod_{p_L <p \le p_M} \left( 1-\frac{\chi_j(p)}{p^{\, s+i t_p}} \right)^{-1} \right|
< \frac{\gamma}{3n} ,
\end{multline*}
has positive lower density.
For these real $t$, we have thus
\begin{equation*}
\max_{s\in C} \left| \sum_{j=1}^n F_j(s+it)-G_j(s) \right|
<\gamma = \min_{s \in C} \left| \sum_{j=1}^n G_j(s) \right| 
\end{equation*}
As $\sum_{j=1}^n G_j(\sigma)=0$ (formula \eqref{thm3eq2}), it follows 
by Rouche's theorem that $F(s+it) = \sum_{j=1}^n F_j(s+it)$ has
at least one zero in $|s-\sigma|<r$. By the positive lower density of these $t$,
we conclude that $N'_F(\sigma_1, \sigma_2, T) \gg_{F, \sigma_1, \sigma_2} T$
for sufficiently large $T$.

\section{Proof of Theorem 3}

We give here only the upper bound for the number 
$N_a^+(1/2+u, +\infty, T)$
of zeros in 
$\frac{1}{2}+u < \R s < +\infty$, $0\le \I s \le T$.
The proof is similar for zeros with negative real part.

Let $\zeta(s,r)$ denote the Hurwitz zeta function.
From Theorem 1 of \cite{KaLa} we have, for $1/2<\sigma < 1$,
\begin{equation}\label{zero}
\begin{split}
\int_0^T |F_a(\sigma + i t)|^2 d t & = \frac{T}{q^{2\sigma}} \sum_{j=1}^q |a_j|^2 \zeta(2\sigma, j/q)
+ O\left(\frac{q^{2-2\sigma} T^{2-2\sigma} \sum_{j=1}^q |a_j|^2}{(2\sigma-1)(1-\sigma)}\right) \\
& = O_{a} \left( \frac{T}{(2\sigma-1)(1-\sigma)}\right),
\end{split}
\end{equation}
since $\zeta(2\sigma,r)= O_r((2\sigma -1)^{-1})$.
By Jensen's inequality,
\begin{equation}\label{5one}
\int_0^T \log |F_a(\sigma+it)| dt \le \frac{T}{2}\log \left(\frac{1}{T} 
\int_0^T |F_a(\sigma + i t)|^2 d t \right)
=O_a(T\log(1/u)),
\end{equation}
for $\sigma=(1+u)/2$, according to \eqref{zero}.

Let $a_m$ be the first nonzero term of the sequence $(a_n)_{n\ge 1}$,
and let $c\ge 2$ be large enough such that, for $\R (s)\ge c$, we have
$F_a(s)=\frac{a_m}{m^s}(1+\theta(s))$ with $|\theta(s)|\le 1/2$.
We apply Littlewood's lemma (see \cite{Tit}, Section 3.8) 
to the rectangle $R$ with vertices $c+ i$, $c+ iT$, $(1+u)/2+ iT$, $(1+u)/2+i$,  to get
\begin{multline*}
2\pi \sum_{\substack{\beta > (1+u)/2 \\ 1<\gamma \le T}} (\beta-(1+u)/2) 
=
\int_{1}^T \log|F_a((1+u)/2+it) | dt
-\int_{1}^T \log|F_a(c+it) | dt \\
+\int_{(u+1)/2}^{c} \arg F_a(\sigma+iT) d\sigma
-\int_{(u+1)/2}^{c} \arg F_a(\sigma+i) d\sigma .
\end{multline*}
The second integral is clearly $O_a(T)$ since $\log|F_a(c+it)| \ll 1$.
Steuding shows on page 302 of \cite{Steu} that $|\arg (1+\theta(\sigma + iT))| \ll \log T$,
if $\sigma$ is from a bounded interval.
Thus the third integral is $O_a(T)$. The last integral is bounded.
Together with \eqref{5one} this shows that
\begin{equation*}
\sum_{\substack{\beta > (1+u)/2 \\ 1<\gamma \le T}} (\beta-(1+u)/2) = O_a(T\log(1/u)).
\end{equation*}
The desired bound now follows from
\begin{equation*}
\frac{u}{2} \, N_a^+(1/2+u, +\infty, T) 
 =\sum_{\substack{\beta > 1/2+u \\ 0\le \gamma \le T}} \frac{u}{2}
\le \sum_{\substack{\beta > (1+u)/2 \\ 0\le \gamma \le T}} (\beta-(1+u)/2)
\ll_a T\log(1/u).
\end{equation*}

\section*{Acknowledgements}

The authors are grateful to Pierre Mazet for the email conversation which 
motivated this study, and to Michel Balazard, who has informed them of the orthogonal
basis of Codec\`{a}, Dvornicich, and Zannier.

\bigskip

\begin{center}
Laboratoire de Probabilit\'{e}s et Mod\`{e}les Al\'{e}atoires,\\
Universit\'{e} Pierre et Marie Curie, 
4 Place Jussieu,
75252 Paris Cedex 05, France \\
e-mail: eric.saias@upmc.fr
\end{center}

\begin{center}
Department of Mathematics, Southern Utah University,\\
Cedar City, UT 84720, USA \\
e-mail: weingartner@suu.edu
\end{center}

\end{document}